\documentclass{amsart}

\usepackage{amssymb}
\usepackage{amsmath}
\usepackage{amsthm}

\usepackage{graphicx}

\usepackage{amscd}
\usepackage[cmtip,all]{xy}

\newtheorem{theorem}{Theorem}[section]

\newtheorem{proposition}[theorem]{Proposition}

\newtheorem{example}[theorem]{Example}

\begin{document}
\title[Rings $K(s)^*(BG)$]{Morava $K(s)^*$-rings of the extensions of $C_p$ by the products of good groups under diagonal action}
\author{Malkhaz Bakuradze}
\address{Iv. Javakhishvili Tbilisi State University, Faculty of Exact and Natural Sciences}
\email{malkhaz.bakuradze@tsu.ge}


\date{}


\begin{abstract}
This note provides a theorem on good groups in the sense of Hopkins-Kuhn-Ravenel \cite{HKR} and some relevant examples.
\end{abstract}

\maketitle{}


\section{Preliminaries and main result}


Let $K(s)^*(-)$, $s>1$, be the $s$-th Morava $K$-theory at prime $p$.
Note that by \cite{JW}, the coefficient ring $K(s)^*(pt)$ is the Laurent polynomial ring in one variable, which is usually denoted in our situation by $\mathbb{F}_p[v_s,v_s^{-1}]$, where $\mathbb{F}_p$ is the field of $p$ elements and $deg(v_s)=-2(p^s-1)$.

In favourable cases $K(s)^*(BG)$ is generated by classes of geometric origin, i.e.,
Euler classes and transfers of Euler classes. In fact, this happens in most if not all successful
calculations to date \cite{SCH1,SCH2,SCH3,SCH4,SCHY}.

Not so surprising, that the examples of calculations with groups $D$,\,$SD$,\,$QD$,\,$Q$,\,$M$ \cite{BV}, \cite{B1,B2} show, that even if the additive structure of $K(s)^*(BG)$ is calculated, the multiplicative structure is still a delicate task. Just outside of the class of p-groups with maximal cyclic subgroup--already for 2-groups of order 32--one is lead to a complicated ring structures \cite{BJ}.

In a previous article \cite{BP}
the author and S. Priddy obtained formulas relating Chern classes of transfer bundles to
transfers of Chern classes. As formal consequences of such formulas one obtains new relations
(as well as often much simpler derivations of old ones), and the hope generally is that these
combined methods prove sufficient.

The main result of this note is Theorem \ref{1} on good groups in the sense of Hopkins-Kuhn-Ravenel as follows.

\medskip

Recall from \cite{HKR} the following definition.

\medskip

a) For a finite group $G$, an element $x\in K(s)^*(BG)$ is good
if it is a transferred Euler class of a complex subrepresentation of $G$, i.e., a class
of the form $Tr^*(e(\rho))$, where $\rho$ is a complex representation of a subgroup $H<G$, $e(\rho)\in K(s)^*(BH)$ is its Euler class (i.e., its top Chern class, this being defined since $K(s)^*$ is a complex oriented theory), and $Tr:BG \to BH$ is the transfer map.

(b) $G$ is called to be good if $K(s)^*(BG)$ is spanned by good elements as a $K(s)^*$–module.


The good groups in the weaker sense, i.e., $K(s)^{odd}=0$, also play a role in the literature
\cite{K},
\cite{SCH1},
\cite{Y3}.

In \cite{K}, Kriz proved a theorem about the Serre spectral sequence \cite{HKR}
\begin{equation*}
E_2=H^*(BC_p,K(s)^*(BH))\Rightarrow K(s)^*(BG)
\end{equation*}
associated to a group extension $1\to H \xrightarrow[]{}G \xrightarrow[]{} C_p\to 1$ for a $p$-group $G$ and normal subgroup $H\lhd G$ with $K(s)^{odd}(BH)=0$. Namely, in \cite{K} Kriz proved, that $K(s)^{odd}(BG)=0$ if and only if the integral Morava
$K$-theory $\tilde{K}(s)^*(BH)$ is a permutation module for the action of $G/H \cong C_p$.
Kriz in \cite{K} and Yagita in \cite{Y3} proved that an extension of an elementary abelian
$p$-group by a cyclic $p$-group satisfies the even-dimensionality conjecture.
For primes greater than 3, groups of $p$-rank $2$  were shown to have even Morava $K$-
theory by Tezuka-Yagita \cite{TY1} and Yagita \cite{Y2}. Furthermore, Yagita also
proved that these groups are generated by transfered Euler classes and are thus
good in the sense of Hopkins-Kuhn-Ravenel. Hunton \cite{HU} used "unitary-like embeddings" to show that if
$K(s)^*(BH)$ is concentrated in even degrees, then so is $K(s)^*(BH\wr C_p)$. An
independent proof of the same fact in the sense HKR was given in \cite{HKR}; Our result is the following

\begin{theorem}
\label{1}
Let $H_i$ and $G_i$ be finite $p$-group, $i=1,\cdots, n$, such that $H_i$ is good and  $G_i$ fits into an extension $1\to H_i\to G_i\to C_p \to 1$.

Let $G$ fits into an extension of the form $1\to H \to G\to C_p \to 1$, with diagonal conjugation action of $C_p$ on $H=H_1 \times\cdots\times H_n $.  Denote by

$Tr^*=Tr^*_{\varrho}: K(s)^*(BH)\to K(s)^*(BG)$, the transfer homomorphism, associated to $p$-covering $\varrho=\varrho(H,G):BH \to BG$,

$Tr_i^*=Tr^*_{\varrho_i}:K(s)^*(BH_i)\to K(s)^*(BG_i)$, the transfer homomorphism, associated to $p$-covering  $\varrho_i=\varrho(H_i,G_i):BH_i \to BG_i$, $i=1,\cdots ,n$,

$\rho_i:BG\to BG_i $, the map, induced by projection $H \to H_i$ on $i$-th factor,

and let $\rho^*$ be the restriction of
$$(\rho_1,\cdots ,\rho_n)^*:K(s)^*(BG_1\times \cdots \times BG_n)\to K(s)^*(BG)$$
on $K(s)^*(BG_1)/Im Tr^*_1 \otimes \cdots \otimes K(s)^*(BG_n)/ImTr^*_n$.  Then

\medskip

i) If $G_i$ are good so is $G$.

ii) $K(s)^*(BG)$ is spanned, as a $K(s)^*(pt)$-module,
by elements of $ImTr^*$ and $Im\rho^*$

\end{theorem}

\begin{proof} Clearly it suffices to consider the case $n=2$. Also ii) implies i). The proof of ii) uses formal properties of the transfer \cite{A}, \cite{KP}, \cite{D}, including the Fr\"{o}benius reciprocity and the
double coset formula. We follow the proof of wreath product theorem (\cite{HKR}, Theorem 3.8).

Consider the spectral sequence $\{E_r^{*,*}(BG)\}$ with
\begin{equation*}
E_2^{*,*}=H^*(BC_p,K(s)^*(BH))\Rightarrow K(s)^*(BG)
\end{equation*}
associated to the group extension
$$1\to H \xrightarrow[]{}G \xrightarrow[]{\pi} C_p\to 1.$$

As $H_i$ are good, we have that the Kunneth isomorphism
$$K(s^*(BH)=K(s)^*(BH_1)\otimes K(s)^*(BH_2),$$
is the $C_p$-module map, induced by diagonal action of $C_p$ on $H_1\times H_2$.

This gives the decompositions
$$[K(s)^*(BH)]^{C_p}=[F]^{C_p}+T, \text{ and }\,\,\,
[H_k]^{C_p}=[F_k]^{C_p}+T_k,\,\,\,k=1,\cdots, n$$
corresponding to the decompositions of  $K(s)^*(BH)$ and $K(s)^*(BH_k)$ into free and trivial $C_p$-modules. Note that
$T=T_1\otimes T_2.$

Here

\begin{equation*}
H^i(BC_p,F)=
\begin{cases}
[F]^{C_p}& \text{for $i=0$} \\
0        & \text{for $i>0$}.
\end{cases}
\end{equation*}

and
$$
H^*(BC_p,T)=H^*(BC_p)\otimes T.
$$

Recall $E_2^{0,*}$ is isomorphic to $[K(s)^*(BH)]^{C_p}$ via $\varrho^*: K(s)^*(BG) \to K(s)^*(BH)$. Also, via $\pi^*$, the spectral sequence $E_2^{*,*}(BG)$ is a module over the Atiyah-Hirzebruch spectral sequence $E_2^{*,*}(BC_p)$ that converges to $K(s)^*(BC_p)$.

The Fr\"{o}benius reciprocity of the transfer says that the composition $\varrho^*Tr^*$ is the trace map $\varrho^*Tr^*=1+t+t^2+\cdots +t^{p-1}$, where $t$ is the generator $t\in G/H\cong C_p$.

Clearly, the trace map $\varrho^*Tr^*$ always maps the good elements of $K(s)^*(BH)$ onto $[F]^{C_p}$.
Hence the group $G$ is good iff $T$ is also covered by $\varrho^*$-images of good elements. Any element $a\in T$ is a couple $a_1\otimes a_2$, where
$a_k\in T_k$. Let $G_k$ be good $i=1,2$. There is $b_k$, a sum of good elements in $K(s)^*(BG_k)/ImTr^*_k$, such that $a_k=\varrho_i^*(b_k)$. It suffices to prove that
$a$ is $\varrho^*$ image of a sum of good elements. To see this note that by formal properties of the transfer the product of good elements $d_1\otimes d_2$ is good (see \cite{HKR}, Lemma 7.2 v)), and its pullback under the homomorphism induced by $(\rho_1,\rho_2):G \to G_1\times G_2$ is a sum of good elements (see \cite{HKR},  Lemma 7.2 iv)). Now $G$ is good by \cite{HKR} Lemma 7.6 as $a=\pi^*\circ(\rho_1,\rho_2)^*(b_1 \otimes b_2)$ and therefore every element of $E_2^{0,*}(BG)$ is a permanent cycle.

\end{proof}

Note that $K(s)^{odd}(BG)=0$ does not implies that $G$ is good in the sense of Hopkins-Kuhn-Ravenel, so that the theorem by Kriz, mentioned above, does not implies Theorem \ref{1} i).



\section{examples}

\begin{example}
\end{example}

As an application recall from  \cite{G-P}, there exist 17 non-isomorphic groups of order $2^{2n+1}$, $n>2$, which can be presented as a semidirect product
$(C_{2^{n}}\times C_{2^{n}})\rtimes C_2$.
Each such group G is given by three generators $\mathbf{a}, \mathbf{b}, \mathbf{c}$ and the defining relations
$\mathbf{a}^{2^{n}}=\mathbf{b}^{2^{n}}=\mathbf{c}^2 = 1$, $\mathbf{ab} = \mathbf{ba}$, $\mathbf{c}^{-1}\mathbf{ac} =\mathbf{a}^i\mathbf{b}^j$, $\mathbf{c}^{-1}\mathbf{bc} =\mathbf{a}^k\mathbf{b}^l$
for some $i, j, k, l \in Z_{2^{n}}$ ($Z_{2^m}$ denotes the ring of residue classes modulo $2^m$).
These groups with diagonal actions are good by Theorem \ref{1}. Moreover taking into account the ring structure of $K(s)^*(BG_k)$, \cite{BV,B1}, where $G_k$ is either, the dihedral, semi-dihedral or quasi-dihedral group, one can read off the generators of $K(s)^*(BG)$ as $K(s)^*(pt)$-algebra.
In particular, $K(s)^*(BH)=K(s)^*[u,v]/(u^{2^{ns}},v^{2^{ns}})$, therefore as a $K(s)^*(pt)$ algebra, $K(s)^*(BG)$ is generated by $Tr^*(u)$, $Tr^*(v)$, $Tr^*(uv)$ and the Euler classes which come from $BG_i$.

\begin{example}
\end{example}

 Let us consider in more details the following example and compare it with examples in \cite{BJ}. It seems that even if the action of $G/H\cong C_2$ is diagonal, i.e., is simpler, the ring structure of $K(s)^*(BG)$ has the same complexity.

The group $G=G_{36}$, with number 36 in the Hall-Senior list \cite{H} of 51 groups of order 32, can be presented as
\begin{equation*}
G_{36}={\langle \mathbf{a},\mathbf{b},\mathbf{c} \mid  \mathbf{a}^4=\mathbf{b}^4=\mathbf{c}^2=[\mathbf{b},\mathbf{c}]=1, \mathbf{a}^{-1}\mathbf{b}\mathbf{a}=\mathbf{b}^{-1}, \mathbf{c}\mathbf{a}\mathbf{c}=\mathbf{a}^{-1} \rangle }.
\end{equation*}

$G_{36}$ contains the maximal abelian subgroup
$H=\langle \mathbf{b},\mathbf{a}^2,\mathbf{c}\rangle \cong C_4 \times C_2\times C_2$.

\medskip

Let $\lambda$, $\mu$ and $\nu$ denote the following complex line bundles over $BH$.
$$\lambda(\mathbf{b})=i, \nu(\mathbf{a}^2)=\mu(\mathbf{c})=-1, \lambda(\mathbf{a}^2)=\lambda(\mathbf{c})=\nu(\mathbf{b})=\nu(\mathbf{c})=\mu(\mathbf{b})=\mu(\mathbf{a}^2)=1.$$

The quotient of $G$ by the centre $Z\cong C_2^2$ is isomorphic to $C_2^3$. The projections on the three factors induce three line bundles.
$\alpha$, $\beta$ and $\gamma$:
$$
\alpha(b)=\beta(c)=\gamma(a)=-1,\,\,\,\alpha(a)=\alpha(c)=\beta(a)=\beta(b)=\gamma(b)=\gamma(c)=1.
$$
Let us denote Chern classes by
$$
x_i=c_i(Ind_H^G(\nu)); \,\,\,y_i=c_i(Ind_H^G)(\lambda),\,\,\,a=c_1(\alpha), \,\,\,b=c_1(\beta),\,\,\,c=c_1(\gamma).
$$

From the definition one reads off that (see \cite{SCH1}) $K(s)^*(BH)=M\otimes N$, where $M=K(s)^*(BC_4)$ and $N=K(s)^*(BC_2\times C_2)$, with the switch action. Both actions define the dihedral group $D_8$, written as $C_4\rtimes C_2$, or $(C_2\times C_2) \rtimes C_2$.

In notations of Theorem \eqref{1} $T_1$ and $T_2$ are generated by $\pi^*$-images of Euler classes respectively $\{y_2^i,by_2^i\}$ and $\{x_2^i,ax_2^i\}$,
$0\leq i\leq  2^{s-1}-1$ \cite{BV}. Thus $T$ is covered by Euler classes and $G$ is good.

\medskip

Let
$$
T=Tr^*(uv),\text{ where } u=c_1(\nu);\,\,\, v=c_1(\lambda).
$$

Then by Fr\"{o}benius reciprocity of the transfer any element in the image of $Tr^*$ can be written in $x_1,y_1,x_2,y_2,T,c$
and one has

\begin{proposition}
\label{G36}
i) $K(s)^*(BG)\cong K(s)^*[a,b,c,x_2,y_2,T]/R$, where the ideal $R$ is generated by

$a^{2^s}$, $b^{2^s}$, $c^{2^s}$,

$c(c+x_1+v_s\sum_{i=1}^{s-1}c^{2^s-2^i}x_2^{2^{i-1}})$,
\,\,$c(c+y_1+v_s\sum_{i=1}^{s-1}c^{2^s-2^i}y_2^{2^{i-1}})$,

$a(a+y_1+v_s\sum_{i=1}^{s-1}a^{2^s-2^i}y_2^{2^{i-1}})$,\,\,\,
$b(b+x_1+v_s\sum_{i=1}^{s-1}b^{2^s-2^i}x_2^{2^{i-1}})$,

$(c+y_1+v_s\sum_{i=1}^{s-1}c^{2^s-2^i}y_2^{2^{i-1}})(b+x_1+v_s\sum_{i=1}^{s-1}b^{2^s-2^i}x_2^{2^{i-1}})+v_sb^{2^s-1}T$,

$(c+x_1+v_s\sum_{i=1}^{s-1}c^{2^s-2^i}x_2^{2^{i-1}})(a+y_1+v_s\sum_{i=1}^{s-1}a^{2^s-2^i}y_2^{2^{i-1}})+v_sa^{2^s-1}T$,

$T^2+Tx_1y_1+x_2y_1(c+y_1+v_s\sum_{i=1}^{s-1}c^{2^s-2^i}y_2^{2^{i-1}})+x_1y_2(c+x_1+v_s\sum_{i=1}^{s-1}c^{2^s-2^i}x_2^{2^{i-1}})$,

$T(b+x_1+v_s\sum_{i=1}^{s-1}b^{2^s-2^i}x_2^{2^{i-1}})+v_sb^{2^s-1}x_2(c+y_1)$,

$T(a+y_1+v_s\sum_{i=1}^{s-1}a^{2^s-2^i}y_2^{2^{i-1}})+v_sa^{2^s-1}y_2(c+x_1)$,\,\,\,$cT$,

$v_s^2x_2^{2^s}+c^2+bc$, \,\,\, $v_s^2y_2^{2^s}+a^2+ac,$

\medskip

where

$x_1=v_s(x_2+v_sx_1x_2^{2^{s-1}})^{2^{s-1}}+b$, \,\,\, $y_1=v_s(y_2+v_sy_1y_2^{2^{s-1}})^{2^{s-1}}+c$.

\medskip

ii) Some other relations are

$a^2c=ac^2,\,\,\,$ $b^2c=bc^2,\,\,\,$ $x_1^{2^{s}}=b^{2^{s-1}}c^{2^{s-1}},\,\,\,$ $y_1^{2^{s}}=a^{2^{s-1}}c^{2^{s-1}}.$
\end{proposition}

To prove that Proposition \ref{G36} indeed gives the ring structure we should (i) check the relations stated, and then to establish two facts:
(ii) the classes defined generate, and (iii) the list of relations is complete.

The proof of ii) uses Theorem \ref{1};  i) and iii) uses the arguments worked out in \cite{BJ} and is left to the reader.


\bibliographystyle{amsplain}

\end{document}